%% file: te.tex
\documentclass[a4paper,10pt, final]{amsart}
\usepackage[utf8]{inputenc}
\usepackage{amsmath, amsthm, amsfonts,amssymb,epsfig,enumerate, tikz, showkeys, fancyhdr , color, tensor}
\usetikzlibrary{arrows}
\usetikzlibrary{matrix}
\usetikzlibrary{patterns}
\usetikzlibrary{decorations.markings}
\usetikzlibrary{decorations.pathmorphing}
\newcounter{res}[section]
\numberwithin{res}{section}
\newtheorem{thm}[res]{Theorem}
\newtheorem{lem}[res]{Lemma}
\newtheorem{prop}[res]{Proposition}

\newtheorem*{firstthm}{First theorem}
\newtheorem*{scndthm}{Second theorem}
\theoremstyle{definition}

\newtheorem{dfn}[res]{Definition}
\newtheorem{req}[res]{Remark}
\newtheorem{exa}[res]{Example}

\renewcommand{\S}{\ensuremath{\mathbb{S}}}

\newcommand{\NN}{\mathbb N}

\newcommand{\RR}{\mathbb R} 

\renewcommand {\epsilon}{\varepsilon}

\renewcommand\marginpar[1]{}

\newcommand\kupc[1]{\ensuremath\left\langle #1 \right\rangle_c}

\newcommand{\imagesfolder}{.}
\newcommand{\ie}{i.~e.~}

\newcommand{\resp}{resp.{} }

\input{foamstikz}

\input{cfoamstikz}

\title{On edge-colorings of planar bicubic graphs}
\author{Louis-Hadrien Robert} \thanks{Supported by a grant from the LabEx IRMIA}
\thanks{This note is part of the author PhD thesis.}
 \address{robert@math.unistra.fr \\ IRMA \\
7, rue René Descartes \\
67084 Strasbourg Cedex \\
France \\}
\date{\today}
\newcommand{\sll}{\ensuremath{\mathfrak{sl}}}
\definecolor{bleufonce}{rgb}{0,0,0.3}
\definecolor{vertfonce}{rgb}{0,0.3,0}
\definecolor{darkred}{rgb}{0.8,0,0}
\definecolor{darkblue}{rgb}{0,0,0.5}
\newcommand{\blue}{\mathrm{\color{darkblue}blue}}
\newcommand{\red}{\mathrm{\color{darkred}red}}
\definecolor{darkgreen}{rgb}{0,0.6,0}
\newcommand{\green}{\mathrm{\color{darkgreen}green}}
\begin{document}
\maketitle
\setcounter{tocdepth}{1}
\tableofcontents

\section*{Introduction, motivations}
\label{sec:intr-motiv}

This note is concerned with edge-colorings of bicubic plane multi-graphs (later on called webs see definition \ref{dfn:web}). It is separated in two sections widely independent. The first one introduces a degree associated with a coloring, and gives formula to compute $\kupc{w}$ the graded number of colorings of a given $w$:

\begin{firstthm}
  We have the following local relations:
  \begin{align*}
   \kupc{\!\websquare[0.4]} &= \kupc{\!\webtwovert[0.4]} + \kupc{\!\webtwohori[0.4]}, \\
   \kupc{\!\webbigon[0.4]}  &= [2] \cdot \kupc{\!\webvert[0.4]},\\
   \kupc{\!\webcircle[0.4]} &= \kupc{\!\webcirclereverse[0.4]} = [3]
  \end{align*}
where $[n]$ stands for $\frac{q^n - q^{-n}}{q-q^{-1}}$ and $q$ is an inderminate.
\end{firstthm}

The second section deals with Kempe equivalence of edge-colorings of webs. We prove the following theorem:

\begin{scndthm}
  Let $w$ be a web, then all edge-colorings of $w$ are Kempe-equivalent.
\end{scndthm}

This fact is actually already known (at least for 3-connected bicubic planar graphs) and due to Fisk~\cite{MR0498207}. The proof given here is different and deals directly with the graphs rather than their line-graphs. Additionnaly we prove that the gradings defined in the first section behaves well under Kempe moves (see~\ref{prop:pm2}).

\subsection*{Motivations}
Webs appear naturally as an efficient combinatorial tool in the representation theory of the Lie algebra $\sll_3$. Web colorings enhance this tool a little more allowing to deal with bases of $\sll_3$-modules, one may compare them with. The first theorem is both a $q$-deformation of a result of Jaeger~\cite{MR1172374} about colorings, and a coloring interpretation of a theorem from Kuperberg~\cite{MR1403861}. 

Kempe equivalence is famous for having been introduce by Kempe in his false proof of the four colors theorem. Since then, it as been widely used in graph coloring theory (we refer to the introduction of \cite{MR2279183} for details). One application of the second theorem is given in \cite{LHRThese} where it helps to compute traces of the identity morphisms in a foam context.

The connection between graphs and representation theory of Hopf algebras is well-known (see~\cite{MR1659228, MR1036112}) and leads to beautiful construction such that the Jones polynomial. However, to the author knowledge, colorings are not very used. he believes that one could take benefit to bring this to areas closer. 

\subsection*{Acknowledgment}
The author thanks Roland Grappe and Lukas Lewark for showing interest for the subject. The author is not familiar with the literature on graph theory and apologize for eventual weird notations, missing citation, etc.

\section{Counting colorings with degrees}
\label{sec:count-color-with}

  \begin{dfn}[Kuperberg, \cite{MR1403861}]\label{dfn:web}
  A \emph{closed web} is a 3-regular oriented multi-graph (with possibly some vertex-less loops) smoothly embedded in $\RR^2$ such that every vertex is either a sink either a source.
\end{dfn}
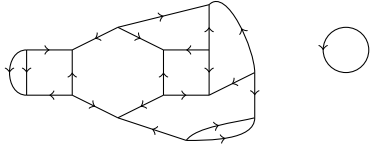
\begin{figure}[!ht]
  \centering
  \begin{tikzpicture}[yscale= 0.6, xscale= 0.6]
    \input{\imagesfolder/sw_ex_closed_web}
  \end{tikzpicture}  
  \caption{Example of a closed web.}
  \label{fig:example-closed-web}
\end{figure}
\begin{req}
  The orientation condition is equivalent to say that the graph is bipartite (by sinks and sources). The vertex-less loops may be a strange notion from the graph theoretic point of view, to prevent this we could have introduced some meaningless 2-valent vertices\footnote{In this case the orientation condition is: around each vertex the flow module 3 is preserved.}, but then the coloring conditions would be a little less natural to express.
\end{req}
The following proposition is the base of the argumentation of the first and the second theorem:
\begin{prop}\label{prop:closed2elliptic}
  Every closed web contains at least a circle, a digon or a square.
\end{prop}
\begin{proof}
Let $w$ be a closed web. One can suppose that the web $w$ is connected for otherwise we consider could an innermost connected component. Suppose that it is not a circle. Then we use that the Euler characteristic  of a connected plane graph is equal to 2:
\[
\chi (w) = \#F(w) - \#E(w) +\#V(w) = 2,
\]
where $F(w)$, $E(w)$ and $V(w)$ are the sets of faces, edges and vertices of $w$.
Each vertex is the end of 3 edges, and each edge has 2 ends, so that $\#V(w)=\frac{2}3 \#E(w)$.
Now $\#F(w) = \sum_{i\in \NN} F_i(w)$ where $F_i(w)$ denotes the number of faces of $w$ with $i$ sides. A web being bipartite, $F_i(w)$ is equal to 0 if $i$ is odd. Each edges belongs to exactly two faces so that $\sum iF_i(w) = 2 \#E(w)$. This gives:
\[
\sum_i F_i(w) -\frac{i}6 F_i(w) =2.
\]
And this shows that at $F_2$ or $F_4$ must be positive.
\end{proof}

We fix $\S$ to be the set of colors $\{\red,\green,\blue\}$. It is endowed with a total order\footnote{This is the frequency order.} on $\S$: $\red<\green<\blue$. For every element $u$ of $\S$, we denote by $\S_{u}$ the set $\S_u\setminus \{u\}$, it is endowed with he order induced by the order of $\S$.

\begin{dfn}
 Let $u$ be a color of  $\mathcal{S}$, a \emph{$\S_u$-colored cycle} $C$ is either an oriented circle embedded in the plane and colored with one of the two colors of $\S_u$, or a connected plane oriented 2-regular $\S_u$-edge-colored graph, such that the vertices are either sinks or sources. 
By $\S_u$-edge-colored we mean that there is an application $\phi_C$ from $E(C)$ the set of edges of $C$ to $\mathcal{S}$ such that two adjacent edges are colored in a different way. The application $\phi_C$ is called the \emph{coloring}.
\end{dfn}
\begin{dfn}
  If $C$ is a $\S_u$-colored cycle we say that it is \emph{positively oriented} if, when reversing the orientations of the edges which are colored with the greatest element of $\S_u$, the result is a counterclockwise oriented cycle. When a $\S_u$-colored cycle is not positively oriented it is \emph{negatively oriented}.
\end{dfn}
\begin{dfn} Let $u$ be an element of $\S$, 
  a \emph{$\S_u$-colored configuration} $D$ is a finite disjoint union of $\S_u$-colored cycles $(C_i)_{i\in I}$.
\end{dfn}
\begin{dfn}
The \emph{degree} $d(D)$ of an $\S_u$-colored configuration $D$ is the algebraic number of oriented cycles, \ie the number of positively oriented cycles minus the number of the negatively oriented cycles.
\end{dfn}
\begin{lem}\label{lem:Ddegree}
  If $u$ is a color of $\S$, denote by $v$ the smallest color of $\S_u$. Suppose that $D_1$ and $D_2$ are two $\S_u$-colored configurations which are the same except in a small ball where:
\[
D_1 = {\vcenter{\hbox{\!
 \begin{tikzpicture}[scale={0.4},decoration={markings, mark=at
     position 0.5 with {\arrow{>}}},postaction={decorate}]
   \draw[dotted] (1.5,1.5) circle (1.414cm);   
   \draw[postaction= {decorate}] (0.5,0.5) .. controls (1,1) and (2,1) .. (2.5, 0.5);
   \draw[postaction= {decorate}] (2.5,2.5) .. controls (2,2) and (1,2) .. (0.5, 2.5);
\node at (1.5,0.4) {$v$};
\node at (1.5,2.6) {$v$};
 \end{tikzpicture}}}}\qquad \textrm{ and}\qquad
D_2= 
\vcenter{\hbox{\!
 \begin{tikzpicture}[scale={0.4},decoration={markings, mark=at
     position 0.5 with {\arrow{>}}},postaction={decorate}]
   \draw[dotted] (1.5,1.5) circle (1.414cm);   
   \draw[postaction= {decorate}] (0.5,0.5) .. controls (1,1) and (1,2) .. (0.5, 2.5);
   \draw[postaction= {decorate}] (2.5,2.5) .. controls (2,2) and (2,1) .. (2.5, 0.5);
\node at (0.4,1.5) {$v$};
\node at (2.6,1.5) {$v$};
 \end{tikzpicture}}}.
\]
Then we have $d(D_2)-d(D_1)=1$.
\end{lem}
\begin{proof}
  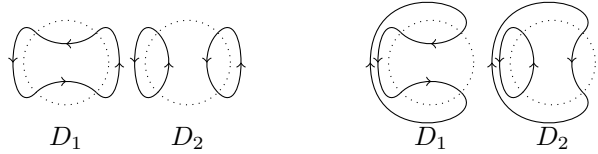
\begin{figure}[!ht]
    \centering
    \begin{tikzpicture}[scale=0.4]
      \input{\imagesfolder/ckb_D1D2}
    \end{tikzpicture}
    \caption{There are two different situations. The dotted circles represent the border of the ball where $D_1$ and $D_2$ are different. It's clear that $d(D_2)-d(D_1)=1$ for both cases.}
    \label{fig:D1D2config}
  \end{figure}
  To show this lemma, it's enough to consider only the connected components of $D_1$ and $D_2$ which meet the ball where $D_1$ and $D_2$ are different. We can as well suppose that these connected components are colored with only one color (which is $v$, the lowest color of $\S_u$). Then there are different situations to check, it's enough to inspect the case where $D_1$ has just one connected component (for otherwise $D_2$ has only one, and reversing the orientations we are back in the situation where $D_1$ has only one connected component). There are essentially two possibilities and they are depicted in figure \ref{fig:D1D2config}.
\end{proof}
\begin{dfn}
  If $w$ is a web and $c$ is an $\S$-edge-coloring of $w$ (we mean a function from the set of edges (and circles) of $w$ to $\S$ such that if two edges are adjacent, they have different colors), we denote by $w_c$ the web together with this coloring. If $u$ is a color of $\mathcal{S}$, we define the \emph{$\S_u$-configuration of $w_c$} denoted by $D_u(w_c)$, the $\S_u$-colored configuration obtained from $w_c$ by deleting all the edges (and circles) with color $u$. We denote by $\mathrm{col}(w)$ the set of all $\S$-edge-colorings of $w$.
\end{dfn}
\begin{dfn}
If $w_c$ is a $\S$-colored web, we define $d_t(w_c)$ the \emph{total degree of $w_c$} to be the sum of the $D_u(w_c)$ when $u$ runs over the colors of $\S$. We define $\kupc{w}$ \emph{the colored Kuperberg bracket of $w$} to be the Laurent polynomial in $q$ defined by the following formula:
\[\kupc{w} =\sum_{c\in\mathrm{col}(w)} \prod _{u\in\S}q^{d(D_u(w_c))} = \sum_{c\in\mathrm{col}(w)} q^{d_t(w_c)}.\]
\end{dfn}
\begin{thm} The colored Kuperberg bracket is multiplicative with respect to direct sum and the following local relations hold:
  \begin{align*}
   \kupc{\!\websquare[0.4]} &= \kupc{\!\webtwovert[0.4]} + \kupc{\!\webtwohori[0.4]}, \\
   \kupc{\!\webbigon[0.4]}  &= [2] \cdot \kupc{\!\webvert[0.4]},\\
   \kupc{\!\webcircle[0.4]} &= \kupc{\!\webcirclereverse[0.4]} = [3]
  \end{align*}
where $[n]$ stands for $\frac{q^n - q^{-n}}{q-q^{-1}}$.
\end{thm}
Note that thanks to proposition~\ref{prop:closed2elliptic}, this theorem gives an explicit way to compute $\kupc{\cdot}$. We divide the proof into three lemmas corresponding to the three given relations.

  \begin{lem}\label{lem:circle}
The colored Kuperberg bracket is multiplicative with respect to disjoint union and we have:
\[\kupc{\!\webcircle[0.3]}=\kupc{\!\webcirclereverse[0.3]}=[3].\]
  \end{lem}
\begin{proof}
The first point is obvious from the definitions. The second point is just a computation. Suppose that the web $w$ is the circle counterclockwise oriented, then there are three different $\S$-edge-colorings for $w$: the circle is either red, green or blue. The details of the computation is given in the table~\ref{tab:degre-circle}. And it's clear from this computations that $\kupc{w}=[3]$. The computation is similar for a clockwise oriented circle.
\begin{table}[ht!]
  \centering
  \begin{tabular}[h]{|c||c|c|c|}\hline
   & $w_{\red}$&$w_{\green}$&$w_{\blue}$\\ \hline \hline
  $D_{\red}(\cdot)$ &   0   &  -1  &  -1   \\ \hline
 $D_{\green}(\cdot)$ &  1    & 0  &  -1   \\ \hline
 $D_{\blue}(\cdot)$ &   1   &  1  &  0  \\ \hline \hline
$d_t(\cdot)$              &   2   &  0  &  -2    \\ \hline
 \end{tabular}
  \caption{Degrees of the colorings of the counterclockwisely oriented circle.}
  \label{tab:degre-circle}
\end{table}
\end{proof}
\begin{lem}\label{lem:digon}
  The colored Kuperberg bracket satisfies the following formula:
\[
\kupc{\!\webbigon[0.5]}=[2]\kupc{\!\webvert[0.5]}.
\]
\end{lem}
\begin{proof}
  We denote by $w$ the web with the digon and $w'$ the web with this digon  replaced by a single strand. The only thing to do is to exhibit two functions $\phi_+$ and $\phi_-$ from $\mathrm{col}(w')$ to $\mathrm{col}(w)$ which are injective, whose images constitute a partition of $\mathrm{col}(w)$ and such that $\phi_+$ increases the total degree of colorings by one and $\phi_-$ decreases the total degree of colorings by one. The functions $\phi_+$ and $\phi_-$ are explicitly described in table \ref{tab:function-phis}.
\begin{table}[ht]
    \centering
    \begin{tabular}[h]{|c||c|c|c|}\hline
      $w'_c$ &  $\cwebvert{red}$   &  $\cwebvert{green}$   &  $\cwebvert{blue}$    \\ \hline\hline
      $w_{\phi_{+}(c)}$ &  $\cwebbigon{red}{green}{blue}$  & $\cwebbigon{green}{red}{blue}$ & $\cwebbigon{blue}{red}{green}$ \\ \hline
      $\Delta^+_\red $& $+1 $ & $ 0$ & $0 $ \\ \hline
      $\Delta^+_\green $& $ 0$ & $ +1 $ & $0$\\ \hline 
      $\Delta^+_\blue $& $ 0$ & $0 $ & $+1 $ \\ \hline \hline    
      $w_{\phi_{-}(c)}$ &   $\cwebbigon{red}{blue}{green}$  & $\cwebbigon{green}{blue}{red}$ & $\cwebbigon{blue}{green}{red}$ \\ \hline
      $\Delta^-_\red $& $-1 $ & $ 0$ & $0 $ \\ \hline
      $\Delta^-_\green $& $ 0$ & $ -1 $ & $0$\\ \hline 
      $\Delta^-_\blue $& $ 0$ & $0 $ & $-1 $ \\ \hline    
    \end{tabular}
    \caption{The functions $\phi_+$ and $\phi_-$. The drawn part is the only non-canonical part, where $\Delta^+_*$ stands for $d(D_*(w_{\phi_+(c)}))-d(D_*(w'_c))$ and $\Delta^-_*$ for $d(D_*(w_{\phi_-(c)}))-d(D_*(w'_c))$. }
    \label{tab:function-phis}
  \end{table}
The injectivity of the maps $\phi_+$ and $\phi_-$ is obvious from their definitions, the fact that their images form a partition of $\mathrm{col}(w)$ as well. The only thing to check is the degree conditions. It's an easy computation, the details are given in table \ref{tab:function-phis}.
\end{proof}
\begin{lem}\label{lem:square}
 The colored Kuperberg bracket satisfies the following formula:
\[
\kupc{\!\websquare[0.5]}=\kupc{\!\webtwohori[0.5]} +\kupc{\!\webtwovert[0.5]}.
\]
\end{lem}
\begin{proof}
We denote by $w$ the web with a square figure, $w'$ the web where the square is replaced by two horizontal strands and $w''$ the web where the square is replaced by two vertical strands. It's enough to describe two injective maps $\phi'$ and $\phi''$ from respectively $\mathrm{col}(w')$ and $\mathrm{col}(w'')$ to $\mathrm{col}(w)$ which preserve the total degree of the colorings and such that their images form a partition of $\mathrm{col}(w)$. We describe them in the tables \ref{tab:square-same-color1}, \ref{tab:square-same-color2}, \ref{tab:square-diff-color2} and \ref{tab:square-diff-color1} depending on if the two strands have the same color or not. In tables \ref{tab:square-diff-color2} and \ref{tab:square-diff-color1}, the missing cases are obtained by a 180 degree rotation. 
\begin{table}[ht!]
  \centering
  \begin{tabular}[h]{|c||c|c|c|} \hline
    $w'_c$ & $\cwebtwohori{red}{red} $ & $\cwebtwohori{green}{green} $ & $\cwebtwohori{blue}{blue} $ \\ \hline
    $w_{\phi'(c)}$ & $\scwebsquare{red}{blue}{green}$ & $\scwebsquare{green}{blue}{red} $ & $ \scwebsquare{blue}{red}{green} $ \\ \hline
    $\Delta'_\red $& $-1 $ & $ \underline{+1}$ & $0 $ \\ \hline
    $\Delta'_\green $& $ \underline{+1}$ & $ -1$ & $-1 $ \\ \hline
    $\Delta'_\blue$& $ 0$ & $0 $ & $\underline{+1} $ \\ \hline
  \end{tabular}
  \caption{Description of the map $\phi'$ for the colorings which give the same color to the two horizontal strands of the web $w'$, where $\Delta'_*$ stands for $d(D_*(w_{\phi'(c)}))-d(D_*(w'_c))$. To compute the underlined values we use lemma~\ref{lem:Ddegree}}
  \label{tab:square-same-color1}
\end{table}
\begin{table}[ht!]
  \centering
  \begin{tabular}[h]{|c||c|c|c|} \hline
    $w''_c$ & $\cwebtwovert{red}{red} $ & $\cwebtwovert{green}{green} $ & $\cwebtwovert{blue}{blue} $ \\ \hline
    $w_{\phi''(c)}$ & $\scwebsquare{red}{green}{blue}$ & $\scwebsquare{green}{red}{blue} $ & $ \scwebsquare{blue}{green}{red} $ \\ \hline
    $\Delta''_\red $& $+1 $ & $ \underline{-1}$ & $0 $ \\ \hline                $\Delta''_\green $& $ \underline{-1}$ & $ +1 $ & $+1$\\ \hline
    $\Delta''_\blue $& $ 0$ & $0 $ & $\underline{-1} $ \\ \hline
  \end{tabular}
  \caption{Description of the map $\phi''$ for the colorings which give the same color to the two vertical strands of the web $w''$, where $\Delta''_*$ stands for $d(D_*(w_{\phi''(c)}))-d(D_*(w''_c))$. To compute the underlined values we use lemma~\ref{lem:Ddegree}}
  \label{tab:square-same-color2}
\end{table}
\begin{table}[ht!]
  \centering
  \begin{tabular}[h]{|c||c|c|c|} \hline
    $w'_c$ & $\cwebtwohori{red}{green} $ & $\cwebtwohori{green}{blue} $ & $\cwebtwohori{red}{blue} $ \\ \hline
    $w_{\phi'(c)}$ & $\dchwebsquare{red}{green}{blue}$ & $\dchwebsquare{green}{blue}{red} $ & $ \dchwebsquare{red}{blue}{green} $ \\ \hline
    $\Delta''_\red $& $0 $ & $ 0$ & $0 $ \\ \hline
    $\Delta''_\green $& $ 0$ & $ 0 $ & $0$\\ \hline 
    $\Delta''_\blue $& $ 0$ & $0 $ & $0 $ \\ \hline
  \end{tabular}
  \caption{Description of the map $\phi'$ for the colorings which give different colors to the two horizontal strands of the web $w'$, where $\Delta'_*$ stands for $d(D_*(w_{\phi'(c)}))-d(D_*(w'_c))$.}
  \label{tab:square-diff-color2}
\end{table}
\begin{table}[ht!]
  \centering
  \begin{tabular}[h]{|c||c|c|c|} \hline
    $w''_c$ & $\cwebtwovert{red}{green} $ & $\cwebtwovert{green}{blue} $ & $\cwebtwovert{red}{blue} $ \\ \hline
    $w_{\phi''(c)}$ & $\dcvwebsquare{red}{green}{blue}$ & $\dcvwebsquare{green}{blue}{red} $ & $ \dcvwebsquare{red}{blue}{green} $ \\ \hline
    $\Delta''_\red $& $0 $ & $ 0$ & $0 $ \\ \hline
    $\Delta''_\green $& $ 0$ & $ 0 $ & $0$\\ \hline
    $\Delta''_\blue         $& $ 0$ & $0 $ & $0 $ \\ \hline
  \end{tabular}
  \caption{Description of the map $\phi''$ for the colorings which give different colors to the two horizontal strands of the web $w''$, where $\Delta''_*$ stands for $d(D_*(w_{\phi''(c)}))-d(D_*(w''_c))$.}
  \label{tab:square-diff-color1}
\end{table}

The injectivity of the maps $\phi'$ and $\phi''$ is clear. The fact that their images form a partition of $\mathrm{col}(w)$ is as well straightforward. 
For every coloring the sum of the three values of $\Delta'$ or $\Delta''$ is equal to zero, and this gives the homogeneity of $\phi'$ and $\phi''$.
\end{proof}

\begin{req}
  One could try to use this new approach of the Kuperberg bracket to have a new approach to the $\sll_3$-homology for links: one can associate to a web the graded vector space with basis the coloring. However, the way to define the differential corresponding to the crossing is not clear. The main obstruction comes from the fact that the degree of a morphism can not be understood locally. A better understanding on how the orientations and the colors interact on the degrees may help. One can define a cheap version of the $\sll_3$-homology which is not quantified and hence categorify only the number of color of a web. It turns out that this invariant does not give any information on the links.
\end{req}

\section{In a web, all colorings are Kempe-equivalent}
\label{sec:web-all-colorings}


\begin{dfn}
Let $w$ be a web, u a color of $\S$ and $C$ is a $\S_u$-bi-colored cycle in $w$. We define the coloring $c'= \tau_C(c)$ to be the coloring of $w$ which is the same as $c$ on all edges which doesn't belong to $C$ and which exchanges the colors of $\S_u$ on the edges in $C$. It's immediate that the cycle $C$ remains $\S_u$-colored in $w_{\tau_C(c)}$  and that we have $c=\tau_C(\tau_C(c))$. 
\end{dfn}
\begin{req}
  The $\tau$-change is of course the Kempe change or K-change in the line graph. 
\end{req}
The following proposition states that the grading behaves well via under $\tau$ providing $\tau$ exchange two consecutive colors.
\begin{prop}\label{prop:pm2}
Let $w_c$ be a closed colored web, $u$ a color of $\S$ different from $\green$ and $C$ a $\S_u$-bi-colored cycle in $w_c$. Then we have 
\[d_t(w_{\tau_C(c)})=
\begin{cases}
d_t(w_c) - 2 & \textrm{if $C$ is positively oriented,}\\
d_t(w_c)+ 2& \textrm{if $C$ is positively oriented.}  
\end{cases}
\]
\end{prop}
Before proving this result we need to introduce a few notions:
\begin{dfn}
  Let $u$ be a color of  $\S$, a \emph{$\S_u$-colored arc} $C$ is a connected oriented $\S_u$-edge-colored graph embedded in $\RR\times[0,1]$ which satisfies the following conditions:
  \begin{itemize}
  \item exactly two vertices of $C$ have valence 1, all the other have valence 2,
  \item the 1-valent vertices are located in $\RR\times\{0,1\}$, and they are the only intersection of $C$ with $\RR\times\{0,1\}$,
  \item the vertices are either sources or sinks,
  \item two adjacent edges have different colors.
  \end{itemize}
As before the \emph{coloring} of $C$ is the function from $E(C)$ to $\S_u$.
\end{dfn}
\begin{dfn}
  Let $u$ be a color of $\S$ and $C$ a $\S_u$-colored arc. We say that $C$ is \emph{$0$-oriented} if one of its 1-valent vertices is on $\RR\times\{0\}$ and the other one on $\RR\times \{1\}$. Now let us reverse the orientations of the edges colored by the highest color of $\S_u$, the arc $C$ is now coherently oriented and it has a tail and a head which have coordinates $(x_t, y_t)$ and $(x_h,y_h)$. It is \emph{positively oriented} if we are in one of the following situations:
  \begin{itemize}
  \item the ordinates are equal to 0 and the abscissae satisfy $x_h< x_t$,
  \item the ordinates are equal to 1 and the abscissae satisfy $x_t< x_h$.
  \end{itemize}
A $\S_u$-colored arc which is neither $0$-oriented nor positively oriented is \emph{negatively oriented}.
\end{dfn}
\begin{dfn}
  A \emph{$\S_u$-path configuration} is a disjoint union of some $\S_u$-colored arcs and $\S_u$-colored cycles which all lie in $\RR\times[0,1]$.
\end{dfn}
Two $\S_u$-path configurations can be composed by glued whenever they are compatible, by stacking them and resizing.
\begin{dfn}
  The \emph{degree $d(D)$} of a $\S_u$-path configuration $D$ is the number given by the formula:
\[
d(D)= p_c-n_c +\frac12(p_a-n_a),
\] where $p_c$ (\resp $n_c$) stands for the number of positively (\resp negatively) oriented $\S_u$-colored cycles and $p_a$ (\resp $n_a$) the number of positively (\resp negatively) oriented $\S_u$-colored arcs.
\end{dfn}
\begin{lem}\label{lem:degrees_adds}
  If $D_b$ and $D_t$ are two $\S_u$-path configurations, so that the composition $D=D_b\circ D_t$ (we mean here that $D_b$ is under $D_t$) is defined then we have $d(D)=d(D_t)+d(D_b)$.
\end{lem}
\begin{proof}
  It comes from the fact that the degree we defined is actually a normalized version of the total curvature of an oriented curve, which is equal to the integral of the signed curvature along the curve. And the Chasles relation gives the result.\marginpar{(find reference ?)} 
\end{proof}
\begin{lem}
  The lemma \ref{lem:Ddegree} holds as well for $\S_u$-path configurations.
\end{lem}
\begin{proof}
  This is clear from the previous lemma.
\end{proof}
\begin{proof}[Proof of proposition~\ref{prop:pm2}]
If $C$ is a vertex-less loop the statement is obvious so we consider the other cases. The coloring $\tau_C(c)$ is denoted by $c'$. As the two situations are symmetric we may assume that $u=\blue$. The $\S_u$-colored configurations $D_\blue(w_c)$ and $D_\blue(w_{c'})$ are the same except for the cycle corresponding to $C$ which is colored in the opposite manner and hence oriented in the opposite manner and we then have: $d(D_\blue(w_c)) = d(D_\blue(w_{c'})) \pm 2$ (depending on how $C$ is oriented). We now look at $D_*=D_{*}(w_c)$ and $D_*'=D_*(w_{c'})$ for $*=\green$ and $\red$. We can perform an isotopy so that the cycle $C$ has all but one of its edges horizontal, and the non-horizontal one, is over the horizontal ones (see figure~\ref{fig:w123}).
\begin{figure}[!ht]
  \centering
  \begin{tikzpicture}[scale =1]
    \input{\imagesfolder/ckb_bicol3part}
  \end{tikzpicture}
  \caption{We decompose $w$ into three web tangles: $w_1$, $w_2$ and $w_3$}
  \label{fig:w123}
\end{figure}
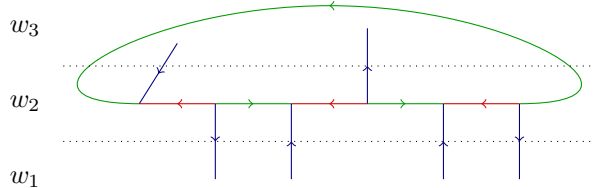
 We now see $w$ as a composition of three web tangles: $w= w^3\circ w^2 \circ w^1$, where $w^2$ contains only the horizontal edges of $C$, two parts of the non-horizontal edge of $C$, the half edges of $w$ which are touching $C$ (they all have color $\blue$ in $w_c$), and vertical arcs far from $C$. It's clear from our hypotheses that $w^1_c= w^1_{c'}$ and $w^3_c$ and $w^3_{c'}$ differ just by one arc which is $\red$ in one and $\green$ in the other one, but as $\blue$ is bigger than this two colors, we have $d(D_{\green}(w^3_c))=d(D_{\green}(w^3_{c'}))\pm 1$ and $d(D_{\red}(w^3_{c'}))=d(D_{\red}(w^3_{c}))\pm 1$, where the $\pm$ signs is the same in both cases. On the other hand we have:
$d(D_{\red}(w^2_c))=d(D_{\green}(w^2_{c'}))$ and $d(D_{\green}(w^2_c))=d(D_{\red}(w^2_{c'}))$. And summing all this equalities together and using the lemma~\ref{lem:degrees_adds} we conclude that:
\[
d(D_{\red}(w_c))+d(D_{\green}(w_{c}))= d(D_{\red}(w_{c'}))+d(D_{\green}(w_{c'})).
\]
This is enough to conclude that $d_t(w_c)=d_t(w_{c'})\pm 2$, depending on how $C$ is oriented. 
\end{proof}
We now define two equivalence relations on $\mathrm{col}(w)$ and the rest of this section will be devoted two show that for a given web $w$ all the colorings are equivalent (see theorem~\ref{thm:coltaueq}).
\begin{dfn}\label{dfn:taueq}
  Let $w$ be a closed web, two colorings $c$ and $c'$ are said to be weakly $\tau$-close if one can find a bi-colored cycle $C$ in $w_c$ so that $c'=\tau_C(c)$. 
They are said to be strongly $\tau$-close if one can find a $\S_\blue$- or a $\S_{\red}$-bi-colored cycle $C$ in $w_c$ so that $c'=\tau_C(c)$.
The weak (\resp strong) $\tau$-equivalence relation is generated by the pairs $(c,c')$ for $c$ and $c'$ (\resp strongly) $\tau$-close from each other.
\end{dfn}
\begin{req}
It's clear from the definition that if $c$ and $c'$ are weakly $\tau$-equivalent, then they are strongly $\tau$-equivalent. The notion of strongly $\tau$-closeness, makes sense in view of proposition \ref{prop:pm2}. Before proving theorem~\ref{thm:coltaueq}, we show that the two different notions of $\tau$-equivalence are the same (see lemma~\ref{lem:tau-eq}).  
\end{req}
\begin{lem}\label{lem:hypesamecol}
  Consider a colored web $w$ and two edges $e_1$ and $e_2$ as in figure~\ref{fig:edgelemma}, then if $e_1$ and $e_2$ are colored differently, then they do not belong to the same bi-colored cycle.
  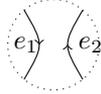
\begin{figure}[!ht]
    \centering
    \begin{tikzpicture}[scale=0.6]
      \input{\imagesfolder/ck_edgelemma}
    \end{tikzpicture}
    \caption{The edges $e_1$ and $e_2$ with their orientations.}
    \label{fig:edgelemma}
  \end{figure}
\end{lem}
\begin{proof}
  If the two edges have different colors, their orientation are not compatible to fit in the same bi-colored cycle.
\end{proof}
\begin{lem}\label{lem:tau-eq}
  If $w$ is a closed web, and $c$ and $c'$ are weakly $\tau$-equivalent colorings of $w$, then $c$ and $c'$ are strongly $\tau$-equivalent. 
\end{lem}
\begin{proof}
  It's enough to show that if two colorings $c$ and $c'$ are weakly $\tau$-close then they are strongly $\tau$-equivalent. As the other cases are straightforward let us suppose that $c'=\tau_C(c)$ with $C$ a $\S_{\green}$-bi-colored cycle. Consider $(C'_i)$ the collection of all the $\S_{\blue}$-bi-colored cycles in $w_c$ which intersect $C$ (it should be at some $\red$ edges). The $C'_i$'s are all disjoint so that we can define $c_1= \tau_{(C'_i)}(c)$ the coloring similar to $c$ except that on all $C'_i$'s, the colors $\red$ and $\green$ are exchanged. It's clear that $c$ and $c_1$ are strongly $\tau$-equivalent. In $w_{c_1}$, $C$ is now $\S_{\red}$-bi-colored, so that we can define $c_2=\tau_C(c_1)$. Now we consider $(C''_i)$ the collection of all the $\S_{\blue}$-bi-colored cycles in $w_{c_2}$ which intersect $C$ (it should be at some $\green$ edges). We define $c_3=\tau_{(C''_i)}(c_2)$ as we defined $c_1$. It's clear that $c_3$ is strongly $\tau$-equivalent to $c_2$ and hence to $c_1$. It's easy to check that $c_3=c'$ because the collection of edges of the $C'_i$'s disjoint from $C$ and the the collection of edges of the $C''_i$'s disjoint from $C$ are equal.
\end{proof}
Because of lemma \ref{lem:tau-eq} we speak of $\tau$-equivalence instead of strong or weak $\tau$-equivalences.
\begin{thm}\label{thm:coltaueq}
  Let $w$ be a closed web then all colorings of $w$ are $\tau$-equivalent.
\end{thm}
\begin{proof}
  We show this by induction on the number of vertices, using the fact that any closed web contains a circle, a digon or a square. Let $w$ be a web and $c_0$ and $c_1$ be two colorings of $w$. 

If $w$ contains a circle $C$, then we consider $w'$ the web $w$ with this circle removed. $c_0$ and $c_1$ induce colorings $c'_0$ and $c'_1$ on $w'$. By induction we know that $c'_0$ and $c'_1$ are $\tau$-equivalent. All the cycles of $w'$ needed to go via $\tau$-moves from $c'_0$ to $c'_1$ can be lifted in $w$. Hence we obtained a coloring $\widetilde{c}_1$, which is $\tau$-equivalent to $c_0$ and equal to $c_1$ everywhere but maybe on the circle, so if necessary we perform $\tau_C$ on $\tilde{c}_1$ where $C$ is seen as a bi-colored circle and we obtain $c_1$ and this show that $c_0$ and $c_1$ are $\tau$-equivalent. 

Suppose now that $w$ contains a digon. We do the same thing, we consider the web $w'$ where the digon figure is replaced by a single strand, and we can play the same game because for any coloring the two strands of the digon form a bicolored cycle. 

Now if $w$ contain a square, thanks to lemma~\ref{lem:hypesamecol}, we may suppose (up to a single $\tau$-move) that $c_0$ has the same color on all the four strands touching the square. We consider the web $w'$ corresponding to one smoothing of the square which is so that $c_0$ and $c_1$ induces naturally colorings on $w'$. We denote them by $c'_0$ and $c'_1$. Let $c$ be a coloring  of $w'$. If the two strands of $w'$ which replace the square are colored in different manners, there is a canonical way to construct a coloring of $w$ and any bi-colored cycle of $w'$ can be lifted up in $w$. If the two strands have the same color, there are two different colorings of $w$ which can induce $c'$ on $w'$, but these two are obviously $\tau$-close. And if $C$ is a bi-colored cycle in $w'_{c'}$ then it can be lifted up in $w$ to a bi-colored cycle endowed with one of this two coloring. So we can play the same game as before: thanks to a path of two by two $\tau$-close coloring from $c'_0$ to $c'_1$, we can construct a (maybe longer) path of two by two $\tau$-close colorings from $c_0$ to $c_1$. And this concludes.
\end{proof}



\begin{exa}
  The dodecahedral graph $G$ does not satisfy the theorem~\ref{thm:coltaueq}.
\end{exa}
\begin{proof}
  Consider the coloring $c$ given by figure~\ref{fig:coldod}. We claim (see figure~\ref{fig:dodpath}) that for every color $u$ in $\S$, it contains only one $\S_u$ bi-colored cycle $C_u$, so that $\tau_{C_u}(c)$ is the same coloring as $c$ but with the two colors of $\S_u$ exchanged, to that the structure of $c$ is not essentially changed.
  \begin{figure}[!ht]
    \centering
    \begin{tikzpicture}[scale=1]
    \input{\imagesfolder/ck_pentagonecol}  
    \end{tikzpicture}
    \caption{The coloring $c$ of the dodecahedral graph}\label{fig:coldod}
  \end{figure}
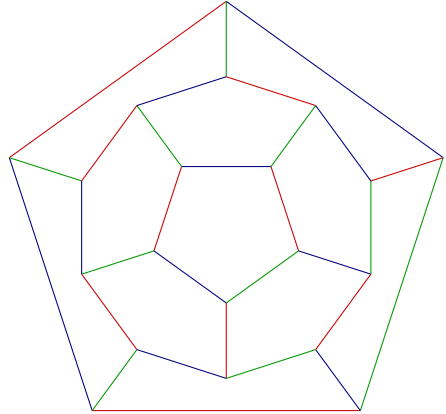
\begin{figure}[!ht]
  \centering
  \begin{tikzpicture}[scale= 0.67]
    \input{\imagesfolder/ck_pentagonepath}
  \end{tikzpicture}
  \caption{The three bi-colored cycle in $G_c$.}
  \label{fig:dodpath}
\end{figure}
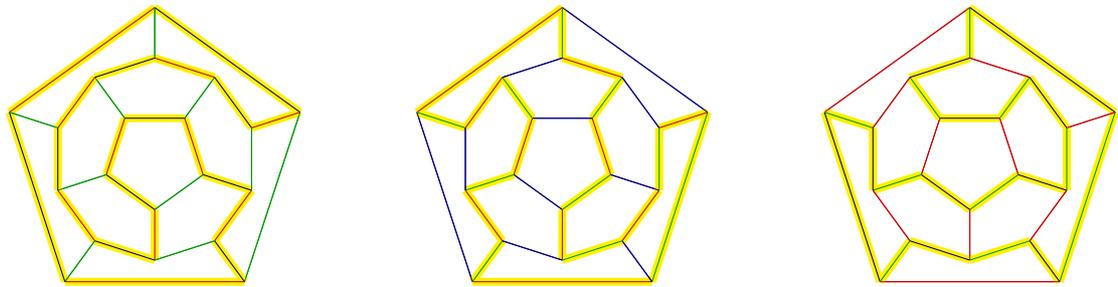
On the other hand, there exist some colorings of $G$ which are essentially different, for example one can  perform a $\frac{2\pi}5$ rotation of the coloring $c$ to obtain a coloring $c'$, therefor $c'$ and $c$ are not $\tau$-equivalent.
\end{proof}
\begin{exa}
  The utility graph (also called $K_{3,3}$) does not satisfied the theorem~\ref{thm:coltaueq}.
\end{exa}
One can apply the same proof as before, figure~\ref{fig:k33}
\begin{figure}[!ht]
  \centering
  \begin{tikzpicture}[scale= 1.5]
    \input{\imagesfolder/ck_nonplanar}
  \end{tikzpicture}
  \caption{A coloring of the graph $K_{3,3}$, and the three bi-colored cycles for this coloring.}
  \label{fig:k33}
\end{figure}
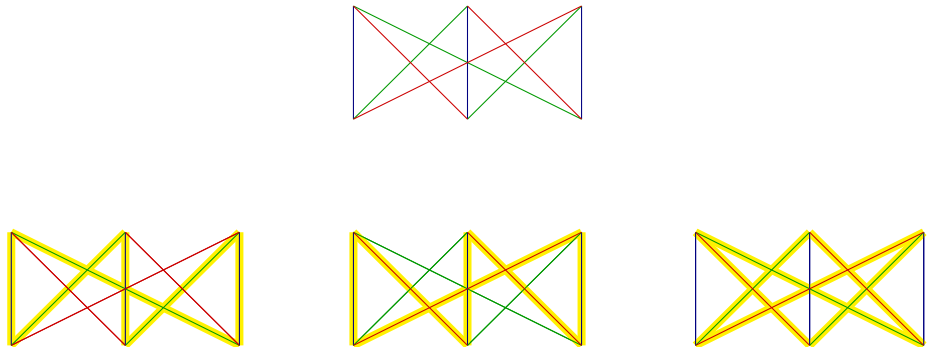
illustrates it.

\bibliographystyle{alpha}
\bibliography{../../biblio}

\end{document}

%% file: foamstikz.tex
\newcommand{\websquare}[1][0.6]{\vcenter{\hbox{\!
 \begin{tikzpicture}[scale={#1},decoration={markings, mark=at
     position 0.5 with {\arrow{>}}},postaction={decorate}]
      \draw[postaction = {decorate}] (0.5,0.5)--(1,1);
      \draw[postaction = {decorate}] (2.5,2.5)--(2,2);
      \draw[postaction = {decorate}] (2,1)--(1,1);
      \draw[postaction = {decorate}] (2,1)--(2,2);
      \draw[postaction = {decorate}] (2,1)--(2.5,0.5);
      \draw[postaction = {decorate}] (1,2)--(2,2);
      \draw[postaction = {decorate}] (1,2)--(1,1);
      \draw[postaction = {decorate}] (1,2)--(0.5,2.5);
 \end{tikzpicture}}}}
\newcommand{\webtwovert}[1][0.6]{
\vcenter{\hbox{\!
 \begin{tikzpicture}[scale={#1},decoration={markings, mark=at
     position 0.5 with {\arrow{>}}},postaction={decorate}]
      \draw[postaction= {decorate}] (0.5,0.5) .. controls (1,1) and (1,2) .. (0.5, 2.5);
      \draw[postaction= {decorate}] (2.5,2.5) .. controls (2,2) and (2,1) .. (2.5, 0.5);
 \end{tikzpicture}}}}

\newcommand{\webtwohori}[1][0.6]{\vcenter{\hbox{\!
 \begin{tikzpicture}[scale={#1},decoration={markings, mark=at
     position 0.5 with {\arrow{>}}},postaction={decorate}]
      \draw[postaction= {decorate}] (0.5,0.5) .. controls (1,1) and (2,1) .. (2.5, 0.5);
      \draw[postaction= {decorate}] (2.5,2.5) .. controls (2,2) and (1,2) .. (0.5, 2.5);
 \end{tikzpicture}}}}

\newcommand{\webvert}[1][0.6]{\vcenter{\hbox{\!
 \begin{tikzpicture}[scale={#1},decoration={markings, mark=at
     position 0.5 with {\arrow{>}}},postaction={decorate}]
      \draw[postaction= {decorate}] (0.5,0.5) -- (0.5, 2.5);
  \end{tikzpicture}}}}

\newcommand{\webbigon}[1][0.6]{\vcenter{\hbox{\!
 \begin{tikzpicture}[scale={#1},decoration={markings, mark=at
     position 0.5 with {\arrow{>}}},postaction={decorate}]
      \draw[postaction= {decorate}] (0.5,0.5) -- (0.5, 1);
      \draw[postaction= {decorate}] (0.5,2) -- (0.5, 2.5);
      \draw[postaction= {decorate}] (0.5,2) ..controls (0,1.5) and (0,1.5).. (0.5, 1);
      \draw[postaction= {decorate}] (0.5,2) ..controls (1,1.5) and (1,1.5).. (0.5, 1);
  \end{tikzpicture}}}}

\newcommand{\webcircle}[1][0.6]{\vcenter{\hbox{\!
 \begin{tikzpicture}[scale={#1},decoration={markings, mark=at
     position 0.5 with {\arrow{>}}},postaction={decorate}]
      \draw[postaction= {decorate}] (1.5,1.5) circle (1cm);
 \end{tikzpicture}}}}
\newcommand{\webcirclereverse}[1][0.6]{\vcenter{\hbox{\!
 \begin{tikzpicture}[scale={#1},decoration={markings, mark=at
     position 0.5 with {\arrow{<}}},postaction={decorate}]
      \draw[postaction= {decorate}] (1.5,1.5) circle (1cm);
 \end{tikzpicture}}}}


%% file: cfoamstikz.tex
\newcommand{\cwebvert}[2][0.8]{\vcenter{\hbox{\!
 \begin{tikzpicture}[scale={#1},decoration={markings, mark=at
     position 0.5 with {\arrow{>}}},postaction={decorate}]
      \draw[postaction= {decorate}, color=dark#2] (0.5,0.5) -- (0.5, 2.5);
      \node at (1,1.5) {\tiny #2};
  \end{tikzpicture}}}}
\newcommand{\cwebbigon}[4][0.9]{\vcenter{\hbox{\!
 \begin{tikzpicture}[scale={#1},decoration={markings, mark=at
     position 0.5 with {\arrow{>}}},postaction={decorate}]
      \draw[postaction= {decorate},color=dark#2] (0.5,0.5) -- (0.5, 1);
      \draw[postaction= {decorate},color=dark#2] (0.5,2) -- (0.5, 2.5);
      \draw[postaction= {decorate},color= dark#3] (0.5,2) ..controls (0,1.5) and (0,1.5).. (0.5, 1);
      \draw[postaction= {decorate},color=dark#4] (0.5,2) ..controls (1,1.5) and (1,1.5).. (0.5, 1);
      \node at (1,0.7) {\tiny #2};
      \node at (1,2.2) {\tiny #2};      
      \node at (-0.5,1.5) {\tiny #3};
      \node at (1.5,1.5) {\tiny #4};
\end{tikzpicture}}}}  
\newcommand{\scwebsquare}[4][0.8]{\vcenter{\hbox{\!
 \begin{tikzpicture}[scale={#1},decoration={markings, mark=at
     position 0.5 with {\arrow{>}}},postaction={decorate}]
      \draw[postaction = {decorate},color =dark#2 ] (0.5,0.5)--(1,1);
      \draw[postaction = {decorate},color =dark#2 ] (2.5,2.5)--(2,2);
      \draw[postaction = {decorate},color =dark#4 ] (2,1)--(1,1);
      \draw[postaction = {decorate},color =dark#3 ] (2,1)--(2,2);
      \draw[postaction = {decorate},color =dark#2 ] (2,1)--(2.5,0.5);
      \draw[postaction = {decorate},color =dark#4 ] (1,2)--(2,2);
      \draw[postaction = {decorate},color =dark#3 ] (1,2)--(1,1);
      \draw[postaction = {decorate},color =dark#2 ] (1,2)--(0.5,2.5);
\node at (0.5,0.4) {\tiny #2 };
\node at (2.5, 0.4) {\tiny #2 };
\node at (0.5, 2.6) {\tiny #2 };
\node at (2.5,2.6) {\tiny #2 };
\node at (1.5, 0.8) {\tiny #4 };
\node at (1.5,2.2) {\tiny #4 };
\node at (0.5,1.5) {\tiny #3 };
\node at (2.5,1.5) {\tiny #3 }; 
\end{tikzpicture}}}}
\newcommand{\dchwebsquare}[4][0.8]{\vcenter{\hbox{\!
 \begin{tikzpicture}[scale={#1},decoration={markings, mark=at
     position 0.5 with {\arrow{>}}},postaction={decorate}]
      \draw[postaction = {decorate},color =dark#2 ] (0.5,0.5)--(1,1);
      \draw[postaction = {decorate},color =dark#3 ] (2.5,2.5)--(2,2);
      \draw[postaction = {decorate},color =dark#3 ] (2,1)--(1,1);
      \draw[postaction = {decorate},color =dark#4 ] (2,1)--(2,2);
      \draw[postaction = {decorate},color =dark#2 ] (2,1)--(2.5,0.5);
      \draw[postaction = {decorate},color =dark#2 ] (1,2)--(2,2);
      \draw[postaction = {decorate},color =dark#4 ] (1,2)--(1,1);
      \draw[postaction = {decorate},color =dark#3 ] (1,2)--(0.5,2.5);
\node at (0.5,0.4) {\tiny #2 };
\node at (2.5, 0.4) {\tiny #2 };
\node at (0.5, 2.6) {\tiny #3 };
\node at (2.5,2.6) {\tiny #3 };
\node at (1.5, 0.8) {\tiny #3 };
\node at (1.5,2.2) {\tiny #2 };
\node at (0.5,1.5) {\tiny #4 };
\node at (2.5,1.5) {\tiny #4 };
 \end{tikzpicture}}}}
\newcommand{\dcvwebsquare}[4][0.8]{\vcenter{\hbox{\!
 \begin{tikzpicture}[scale={#1},decoration={markings, mark=at
     position 0.5 with {\arrow{>}}},postaction={decorate}]
      \draw[postaction = {decorate},color =dark#2 ] (0.5,0.5)--(1,1);
      \draw[postaction = {decorate},color =dark#3 ] (2.5,2.5)--(2,2);
      \draw[postaction = {decorate},color =dark#4 ] (2,1)--(1,1);
      \draw[postaction = {decorate},color =dark#2 ] (2,1)--(2,2);
      \draw[postaction = {decorate},color =dark#3 ] (2,1)--(2.5,0.5);
      \draw[postaction = {decorate},color =dark#4 ] (1,2)--(2,2);
      \draw[postaction = {decorate},color =dark#3 ] (1,2)--(1,1);
      \draw[postaction = {decorate},color =dark#2 ] (1,2)--(0.5,2.5);
\node at (0.5,0.4) {\tiny #2 };
\node at (2.5, 0.4) {\tiny #3 };
\node at (0.5, 2.6) {\tiny #2 };
\node at (2.5,2.6) {\tiny #3 };
\node at (1.5, 0.8) {\tiny #4 };
\node at (1.5,2.2) {\tiny #4 };
\node at (0.5,1.5) {\tiny #3 };
\node at (2.5,1.5) {\tiny #2 };
 \end{tikzpicture}}}}
\newcommand{\cwebtwovert}[3][0.8]{
\vcenter{\hbox{\!
 \begin{tikzpicture}[scale={#1},decoration={markings, mark=at
     position 0.5 with {\arrow{>}}},postaction={decorate}]
      \draw[postaction= {decorate},color =dark#2 ] (0.5,0.5) .. controls (1,1) and (1,2) .. (0.5, 2.5);
      \draw[postaction= {decorate},color =dark#3 ] (2.5,2.5) .. controls (2,2) and (2,1) .. (2.5, 0.5);
\draw[left] (1,1.5)node {\tiny #2 };
\draw[right] (2,1.5) node {\tiny #3};
 \end{tikzpicture}}}}
\newcommand{\cwebtwohori}[3][0.8]{\vcenter{\hbox{\!
 \begin{tikzpicture}[scale={#1},decoration={markings, mark=at
     position 0.5 with {\arrow{>}}},postaction={decorate}]
      \draw[postaction= {decorate},color =dark#2 ] (0.5,0.5) .. controls (1,1) and (2,1) .. (2.5, 0.5);
      \draw[postaction= {decorate},color =dark#3 ] (2.5,2.5) .. controls (2,2) and (1,2) .. (0.5, 2.5);
\draw[above] (1.5, 1) node {\tiny #2 };
\draw[below] (1.5,2) node {\tiny #3 };
 \end{tikzpicture}}}}

%% file: sw_ex_closed_web.tex
\begin{scope}
   [yscale = {1}, xscale={1},decoration={markings, mark=at
     position 0.5 with {\arrow{>}}},postaction={decorate}]
\coordinate (A) at (0,0);
\coordinate (B) at (1,0);
\coordinate (C) at (2,-0.5);
\coordinate (D) at (3,0);
\coordinate (E) at (4,0);
\coordinate (F) at (5,0.5);
\coordinate (A1) at (0,1);
\coordinate (B1) at (1,1);
\coordinate (C1) at (2,1.5);
\coordinate (D1) at (3,1);
\coordinate (E1) at (4,1);
\coordinate (F1) at (4,2);
\coordinate (G) at (5,-0.5);
\coordinate (H) at (3.5,-1);

\draw[postaction=decorate] (A1) --(A);
\draw[postaction=decorate] (A1) .. controls +(-0.5,0)  and  +(-0.5,0).. (A);
\draw[postaction=decorate] (A1) -- (B1);
\draw[postaction=decorate] (B)-- (A);
\draw[postaction=decorate] (B)--(B1);
\draw[postaction=decorate] (B)--(C);
\draw[postaction=decorate] (C1)--(B1);
\draw[postaction=decorate] (C1)--(D1);
\draw[postaction=decorate] (C1)--(F1);
\draw[postaction=decorate] (D)--(C);
\draw[postaction=decorate] (D)--(D1);
\draw[postaction=decorate] (D)--(E);
\draw[postaction=decorate] (F) -- (G);
\draw[postaction=decorate] (F)-- (E);
\draw[postaction=decorate] (F) .. controls +(0,0.5) and  +(0.4,0.4) .. (F1);
\draw[postaction=decorate] (E1)--(F1);
\draw[postaction=decorate] (E1)--(D1);
\draw[postaction=decorate] (E1)--(E);
\draw[postaction=decorate] (H)--(C);
\draw[postaction=decorate] (H) .. controls +(0.5,0) and  +(0,-0.5) .. (G);
\draw[postaction=decorate] (H) .. controls +(0.2,0.3) and  +(-0.4,-0.1) .. (G);
\draw[postaction=decorate] (7,1) circle (0.5cm);
\end{scope}

%% file: ckb_D1D2.tex
\begin{scope}[scale={1},decoration={markings, mark=at
     position 0.5 with {\arrow{>}}},postaction={decorate}]
      \draw[postaction= {decorate}] (0.5,0.5) .. controls (1,1) and (2,1) .. (2.5, 0.5);
      \draw[postaction= {decorate}] (2.5,2.5) .. controls (2,2) and (1,2) .. (0.5, 2.5);
      \draw[postaction= {decorate}] (0.5,2.5) .. controls +(-1,1) and +(-1,-1) .. (0.5, 0.5);
      \draw[postaction= {decorate}] (2.5,0.5) .. controls +(1,-1) and +(1,1) .. (2.5, 2.5);
\draw[dotted] (1.5,1.5) circle (1.4);
\node at (1.5, -1) {$D_1$}; 
\end{scope}
 \begin{scope}[xshift = 12cm, scale={1},decoration={markings, mark=at
     position 0.5 with {\arrow{>}}},postaction={decorate}]
      \draw[postaction= {decorate}] (0.5,0.5) .. controls (1,1) and (2,1) .. (2.5, 0.5);
      \draw[postaction= {decorate}] (2.5,2.5) .. controls (2,2) and (1,2) .. (0.5, 2.5);
      \draw[postaction= {decorate}] (0.5,2.5) .. controls +(-1,1) and +(-1,-1) .. (0.5, 0.5);
      \draw[postaction= {decorate}] (2.5,0.5) .. controls +(1,-1) and +(0,-3).. (-0.5,1.5)..   controls +(0,3) and +(1,1) .. (2.5, 2.5);
\draw[dotted] (1.5,1.5) circle (1.4); 
\node at (1.5, -1) {$D_1$};
 \end{scope}.
\begin{scope}[xshift= 4cm,scale={1},decoration={markings, mark=at
     position 0.5 with {\arrow{>}}},postaction={decorate}]
      \draw[postaction= {decorate}] (0.5,0.5) .. controls (1,1) and (1,2) .. (0.5, 2.5);
      \draw[postaction= {decorate}] (2.5,2.5) .. controls (2,2) and (2,1) .. (2.5, 0.5);
      \draw[postaction= {decorate}] (0.5,2.5) .. controls +(-1,1) and +(-1,-1) .. (0.5, 0.5);
      \draw[postaction= {decorate}] (2.5,0.5) .. controls +(1,-1) and +(1,1) .. (2.5, 2.5);
\node at (1.5, -1) {$D_2$};
\draw[dotted] (1.5,1.5) circle (1.4); 
\end{scope}
 \begin{scope}[xshift=16cm,scale={1},decoration={markings, mark=at
     position 0.5 with {\arrow{>}}},postaction={decorate}]
      \draw[postaction= {decorate}] (0.5,0.5) .. controls (1,1) and (1,2) .. (0.5, 2.5);
      \draw[postaction= {decorate}] (2.5,2.5) .. controls (2,2) and (2,1) .. (2.5, 0.5);
      \draw[postaction= {decorate}] (0.5,2.5) .. controls +(-1,1) and +(-1,-1) .. (0.5, 0.5);
      \draw[postaction= {decorate}] (2.5,0.5) .. controls +(1,-1) and +(0,-3).. (-0.5,1.5)..   controls +(0,3) and +(1,1) .. (2.5, 2.5);
\draw[dotted] (1.5,1.5) circle (1.4); 
\node at (1.5, -1) {$D_2$};
 \end{scope}.

%% file: ckb_bicol3part.tex
\begin{scope}[scale={1},decoration={markings, mark=at
     position 0.5 with {\arrow{>}}},postaction={decorate}]
   \draw[postaction= {decorate}, color = darkgreen] (0,0) -- +(1, 0);
   \draw[postaction= {decorate}, color = darkred] (0,0) -- +(-1, 0);
   \draw[postaction= {decorate}, color = darkblue] (0,0) -- +(0,-1);
   \draw[postaction= {decorate}, color = darkblue] (-0.5, 0.8)-- (-1,0);
   \draw[postaction= {decorate}, color = darkgreen] (2,0) -- +(1, 0);
   \draw[postaction= {decorate}, color = darkred] (2,0) -- +(-1, 0);
   \draw[postaction= {decorate}, color = darkblue] (2,0) -- +(0,+1);
   \draw[postaction= {decorate}, color = darkblue] (1, -1)-- (1,0);
   \draw[postaction= {decorate}, color = darkblue] (4,0) -- +(0,-1);
   \draw[postaction= {decorate}, color = darkblue] (3, -1)-- (3,0);
   \draw[postaction= {decorate}, color = darkred] (4,0) -- +(-1, 0);
   \draw[postaction= {decorate}, color = darkgreen] (4,0) .. controls +(2,0) and(4,1.3) .. (1.5,1.3) .. controls (-1,1.3) and +(-2, 0) .. (-1,0);
   \draw[dotted] (-2,0.5)-- (5,0.5);
   \draw[dotted] (-2,-0.5)-- (5,-0.5);
   \node at (-2.5,0) {$w_2$}; 
   \node at (-2.5,-1) {$w_1$}; 
   \node at (-2.5,1) {$w_3$}; 
\end{scope}

%% file: ck_edgelemma.tex
\begin{scope}[yscale = {1}, xscale={1},decoration={markings, mark=at
     position 0.5 with {\arrow{>}}},postaction={decorate}, xshift=0cm]
\draw[dotted] (0,0) circle (1cm);
\draw[postaction={decorate}] (-50:1) node[midway, label=right:{~~$e_2$}] {} .. controls (0.2,0) and (0.2,0) .. (50:1);
\draw[postaction={decorate}] (130:1)  node[midway, label=left:{$e_1$~~}] {} .. controls (-0.2,0) and (-0.2,0) .. (-130:1);
\end{scope}

%% file: ck_pentagonecol.tex
\begin{scope}[yscale = {1}, xscale={1},decoration={markings, mark=at
     position 0.5 with {\arrow{>}}},postaction={decorate}, xshift=0cm, yshift= 8cm]
\draw[color=darkblue] (18:3) -- (90:3); 
\draw[color=darkred] (90:3) -- (162:3);
\draw[color=darkblue](162:3) -- ( 234:3);
\draw[color=darkred]( 234:3) -- (306:3);
\draw[color = darkgreen](306:3) -- (18:3);
\draw[color=darkred] (18:3) -- (18:2); 
\draw[color=darkgreen] (90:3) -- (90:2);
\draw[color=darkgreen](162:3) -- (162:2);
\draw[color=darkgreen]( 234:3) -- (234:2);
\draw[color = darkblue](306:3) -- (306:2);
\draw[color=darkred] (-18:1) -- (54:1); 
\draw[color=darkblue] (54:1) -- (126:1);
\draw[color=darkred](126:1) -- ( 198:1);
\draw[color=darkblue]( 198:1) -- (270:1);
\draw[color = darkgreen](270:1) -- (-18:1);
\draw[color=darkblue] (-18:1) -- (-18:2);
\draw[color=darkgreen] (54:1) -- (54:2);
\draw[color=darkgreen](126:1) -- (126:2);
\draw[color=darkgreen]( 198:1) -- ( 198:2);
\draw[color = darkred](270:1) -- (270:2);
\draw[color=darkgreen] (-18:2) -- (18:2); 
\draw[color=darkblue] (18:2) -- (54:2);
\draw[color=darkred](54:2) -- (90:2);
\draw[color=darkblue]( 90:2) -- (126:2);
\draw[color = darkred](126:2) -- (162:2);
\draw[color=darkblue] (162:2) -- (198:2); 
\draw[color=darkred] (198:2) -- (234:2);
\draw[color=darkblue](234:2) -- (270:2);
\draw[color=darkgreen]( 270:2) -- (306:2);
\draw[color = darkred](306:2) -- (-18:2);
\end{scope}

%% file: ck_pentagonepath.tex
\begin{scope}[yscale = {1}, xscale={1},decoration={markings, mark=at
     position 0.5 with {\arrow{>}}},postaction={decorate}, xshift=0cm]
\draw[color=yellow, line width=3pt] (18:3) -- (90:3); 
\draw[color=yellow, line width=3pt] (90:3) -- (162:3);
\draw[color=yellow, line width=3pt](162:3) -- ( 234:3);
\draw[color=yellow, line width=3pt]( 234:3) -- (306:3);
\draw[color = darkgreen](306:3) -- (18:3);
\draw[color=yellow, line width=3pt] (18:3) -- (18:2); 
\draw[color=darkgreen] (90:3) -- (90:2);
\draw[color=darkgreen](162:3) -- (162:2);
\draw[color=darkgreen]( 234:3) -- (234:2);
\draw[color = yellow, line width=3pt](306:3) -- (306:2);
\draw[color=yellow, line width=3pt] (-18:1) -- (54:1); 
\draw[color=yellow, line width=3pt] (54:1) -- (126:1);
\draw[color=yellow, line width=3pt](126:1) -- ( 198:1);
\draw[color=yellow, line width=3pt]( 198:1) -- (270:1);
\draw[color = darkgreen](270:1) -- (-18:1);
\draw[color=yellow, line width=3pt] (-18:1) -- (-18:2);
\draw[color=darkgreen] (54:1) -- (54:2);
\draw[color=darkgreen](126:1) -- (126:2);
\draw[color=darkgreen]( 198:1) -- ( 198:2);
\draw[color = yellow, line width=3pt](270:1) -- (270:2);
\draw[color=darkgreen] (-18:2) -- (18:2); 
\draw[color=yellow, line width=3pt] (18:2) -- (54:2);
\draw[color=yellow, line width=3pt](54:2) -- (90:2);
\draw[color=yellow, line width=3pt]( 90:2) -- (126:2);
\draw[color = yellow, line width=3pt](126:2) -- (162:2);
\draw[color=yellow, line width=3pt] (162:2) -- (198:2); 
\draw[color=yellow, line width=3pt] (198:2) -- (234:2);
\draw[color=yellow, line width=3pt](234:2) -- (270:2);
\draw[color=darkgreen]( 270:2) -- (306:2);
\draw[color = yellow, line width=3pt](306:2) -- (-18:2);
\end{scope}

\begin{scope}[yscale = {1}, xscale={1},decoration={markings, mark=at
     position 0.5 with {\arrow{>}}},postaction={decorate}, xshift=0cm]
\draw[color=darkblue] (18:3) -- (90:3); 
\draw[color=darkred] (90:3) -- (162:3);
\draw[color=darkblue](162:3) -- ( 234:3);
\draw[color=darkred]( 234:3) -- (306:3);
\draw[color = darkgreen](306:3) -- (18:3);
\draw[color=darkred] (18:3) -- (18:2); 
\draw[color=darkgreen] (90:3) -- (90:2);
\draw[color=darkgreen](162:3) -- (162:2);
\draw[color=darkgreen]( 234:3) -- (234:2);
\draw[color = darkblue](306:3) -- (306:2);
\draw[color=darkred] (-18:1) -- (54:1); 
\draw[color=darkblue] (54:1) -- (126:1);
\draw[color=darkred](126:1) -- ( 198:1);
\draw[color=darkblue]( 198:1) -- (270:1);
\draw[color = darkgreen](270:1) -- (-18:1);
\draw[color=darkblue] (-18:1) -- (-18:2);
\draw[color=darkgreen] (54:1) -- (54:2);
\draw[color=darkgreen](126:1) -- (126:2);
\draw[color=darkgreen]( 198:1) -- ( 198:2);
\draw[color = darkred](270:1) -- (270:2);
\draw[color=darkgreen] (-18:2) -- (18:2); 
\draw[color=darkblue] (18:2) -- (54:2);
\draw[color=darkred](54:2) -- (90:2);
\draw[color=darkblue]( 90:2) -- (126:2);
\draw[color = darkred](126:2) -- (162:2);
\draw[color=darkblue] (162:2) -- (198:2); 
\draw[color=darkred] (198:2) -- (234:2);
\draw[color=darkblue](234:2) -- (270:2);
\draw[color=darkgreen]( 270:2) -- (306:2);
\draw[color = darkred](306:2) -- (-18:2);
\end{scope}

\begin{scope}[yscale = {1}, xscale={1},decoration={markings, mark=at
     position 0.5 with {\arrow{>}}},postaction={decorate}, xshift=8cm]
\draw[color=darkblue] (18:3) -- (90:3); 
\draw[color=yellow, line width=3pt] (90:3) -- (162:3);
\draw[color=darkblue](162:3) -- ( 234:3);
\draw[color=yellow, line width=3pt]( 234:3) -- (306:3);
\draw[color = yellow, line width=3pt](306:3) -- (18:3);
\draw[color=yellow, line width=3pt] (18:3) -- (18:2); 
\draw[color=yellow, line width=3pt] (90:3) -- (90:2);
\draw[color=yellow, line width=3pt](162:3) -- (162:2);
\draw[color=yellow, line width=3pt]( 234:3) -- (234:2);
\draw[color = darkblue](306:3) -- (306:2);
\draw[color=yellow, line width=3pt] (-18:1) -- (54:1); 
\draw[color=darkblue] (54:1) -- (126:1);
\draw[color=yellow, line width=3pt](126:1) -- ( 198:1);
\draw[color=darkblue]( 198:1) -- (270:1);
\draw[color = yellow, line width=3pt](270:1) -- (-18:1);
\draw[color=darkblue] (-18:1) -- (-18:2);
\draw[color=yellow, line width=3pt] (54:1) -- (54:2);
\draw[color=yellow, line width=3pt](126:1) -- (126:2);
\draw[color=yellow, line width=3pt]( 198:1) -- ( 198:2);
\draw[color = yellow, line width=3pt](270:1) -- (270:2);
\draw[color=yellow, line width=3pt] (-18:2) -- (18:2); 
\draw[color=darkblue] (18:2) -- (54:2);
\draw[color=yellow, line width=3pt](54:2) -- (90:2);
\draw[color=darkblue]( 90:2) -- (126:2);
\draw[color = yellow, line width=3pt](126:2) -- (162:2);
\draw[color=darkblue] (162:2) -- (198:2); 
\draw[color=yellow, line width=3pt] (198:2) -- (234:2);
\draw[color=darkblue](234:2) -- (270:2);
\draw[color=yellow, line width=3pt]( 270:2) -- (306:2);
\draw[color = yellow, line width=3pt](306:2) -- (-18:2);
\end{scope}

\begin{scope}[yscale = {1}, xscale={1},decoration={markings, mark=at
     position 0.5 with {\arrow{>}}},postaction={decorate}, xshift=8cm]
\draw[color=darkblue] (18:3) -- (90:3); 
\draw[color=darkred] (90:3) -- (162:3);
\draw[color=darkblue](162:3) -- ( 234:3);
\draw[color=darkred]( 234:3) -- (306:3);
\draw[color = darkgreen](306:3) -- (18:3);
\draw[color=darkred] (18:3) -- (18:2); 
\draw[color=darkgreen] (90:3) -- (90:2);
\draw[color=darkgreen](162:3) -- (162:2);
\draw[color=darkgreen]( 234:3) -- (234:2);
\draw[color = darkblue](306:3) -- (306:2);
\draw[color=darkred] (-18:1) -- (54:1); 
\draw[color=darkblue] (54:1) -- (126:1);
\draw[color=darkred](126:1) -- ( 198:1);
\draw[color=darkblue]( 198:1) -- (270:1);
\draw[color = darkgreen](270:1) -- (-18:1);
\draw[color=darkblue] (-18:1) -- (-18:2);
\draw[color=darkgreen] (54:1) -- (54:2);
\draw[color=darkgreen](126:1) -- (126:2);
\draw[color=darkgreen]( 198:1) -- ( 198:2);
\draw[color = darkred](270:1) -- (270:2);
\draw[color=darkgreen] (-18:2) -- (18:2); 
\draw[color=darkblue] (18:2) -- (54:2);
\draw[color=darkred](54:2) -- (90:2);
\draw[color=darkblue]( 90:2) -- (126:2);
\draw[color = darkred](126:2) -- (162:2);
\draw[color=darkblue] (162:2) -- (198:2); 
\draw[color=darkred] (198:2) -- (234:2);
\draw[color=darkblue](234:2) -- (270:2);
\draw[color=darkgreen]( 270:2) -- (306:2);
\draw[color = darkred](306:2) -- (-18:2);
\end{scope}

\begin{scope}[yscale = {1}, xscale={1},decoration={markings, mark=at
     position 0.5 with {\arrow{>}}},postaction={decorate}, xshift=16cm]
\draw[color=yellow, line width=3pt] (18:3) -- (90:3); 
\draw[color=darkred] (90:3) -- (162:3);
\draw[color=yellow, line width=3pt](162:3) -- ( 234:3);
\draw[color=darkred]( 234:3) -- (306:3);
\draw[color = yellow, line width=3pt](306:3) -- (18:3);
\draw[color=darkred] (18:3) -- (18:2); 
\draw[color=yellow, line width=3pt] (90:3) -- (90:2);
\draw[color=yellow, line width=3pt](162:3) -- (162:2);
\draw[color=yellow, line width=3pt]( 234:3) -- (234:2);
\draw[color = yellow, line width=3pt](306:3) -- (306:2);
\draw[color=darkred] (-18:1) -- (54:1); 
\draw[color=yellow, line width=3pt] (54:1) -- (126:1);
\draw[color=darkred](126:1) -- ( 198:1);
\draw[color=yellow, line width=3pt]( 198:1) -- (270:1);
\draw[color = yellow, line width=3pt](270:1) -- (-18:1);
\draw[color=yellow, line width=3pt] (-18:1) -- (-18:2);
\draw[color=yellow, line width=3pt] (54:1) -- (54:2);
\draw[color=yellow, line width=3pt](126:1) -- (126:2);
\draw[color=yellow, line width=3pt]( 198:1) -- ( 198:2);
\draw[color = darkred](270:1) -- (270:2);
\draw[color=yellow, line width=3pt] (-18:2) -- (18:2); 
\draw[color=yellow, line width=3pt] (18:2) -- (54:2);
\draw[color=darkred](54:2) -- (90:2);
\draw[color=yellow, line width=3pt]( 90:2) -- (126:2);
\draw[color = darkred](126:2) -- (162:2);
\draw[color=yellow, line width=3pt] (162:2) -- (198:2); 
\draw[color=darkred] (198:2) -- (234:2);
\draw[color=yellow, line width=3pt](234:2) -- (270:2);
\draw[color=yellow, line width=3pt]( 270:2) -- (306:2);
\draw[color = darkred](306:2) -- (-18:2);
\end{scope}
\begin{scope}[yscale = {1}, xscale={1},decoration={markings, mark=at
     position 0.5 with {\arrow{>}}},postaction={decorate}, xshift=16cm]
\draw[color=darkblue] (18:3) -- (90:3); 
\draw[color=darkred] (90:3) -- (162:3);
\draw[color=darkblue](162:3) -- ( 234:3);
\draw[color=darkred]( 234:3) -- (306:3);
\draw[color = darkgreen](306:3) -- (18:3);
\draw[color=darkred] (18:3) -- (18:2); 
\draw[color=darkgreen] (90:3) -- (90:2);
\draw[color=darkgreen](162:3) -- (162:2);
\draw[color=darkgreen]( 234:3) -- (234:2);
\draw[color = darkblue](306:3) -- (306:2);
\draw[color=darkred] (-18:1) -- (54:1); 
\draw[color=darkblue] (54:1) -- (126:1);
\draw[color=darkred](126:1) -- ( 198:1);
\draw[color=darkblue]( 198:1) -- (270:1);
\draw[color = darkgreen](270:1) -- (-18:1);
\draw[color=darkblue] (-18:1) -- (-18:2);
\draw[color=darkgreen] (54:1) -- (54:2);
\draw[color=darkgreen](126:1) -- (126:2);
\draw[color=darkgreen]( 198:1) -- ( 198:2);
\draw[color = darkred](270:1) -- (270:2);
\draw[color=darkgreen] (-18:2) -- (18:2); 
\draw[color=darkblue] (18:2) -- (54:2);
\draw[color=darkred](54:2) -- (90:2);
\draw[color=darkblue]( 90:2) -- (126:2);
\draw[color = darkred](126:2) -- (162:2);
\draw[color=darkblue] (162:2) -- (198:2); 
\draw[color=darkred] (198:2) -- (234:2);
\draw[color=darkblue](234:2) -- (270:2);
\draw[color=darkgreen]( 270:2) -- (306:2);
\draw[color = darkred](306:2) -- (-18:2);
\end{scope}

%% file: ck_nonplanar.tex
\begin{scope}[yscale = {1}, xscale={1},decoration={markings, mark=at
     position 0.5 with {\arrow{>}}},postaction={decorate}, xshift=1cm]
\coordinate (A0) at (0,0);
\coordinate (A1) at (1,0);
\coordinate (A2) at (2,0);
\coordinate (B0) at (0,1);
\coordinate (B1) at (1,1);
\coordinate (B2) at (2,1);
\draw[yellow, line width = 3pt] (A0) -- (B0);
\draw[yellow, line width = 3pt] (A1) -- (B1);
\draw[yellow, line width = 3pt] (A2) -- (B2);
\draw[yellow, line width = 3pt] (A0) -- (B1);
\draw[yellow, line width = 3pt] (A1) -- (B2);
\draw[yellow, line width = 3pt] (A2) -- (B0);
\draw[darkred] (A0) -- (B2);
\draw[darkred] (A1) -- (B0);
\draw[darkred] (A2) -- (B1);
\end{scope}
\begin{scope}[yscale = {1}, xscale={1},decoration={markings, mark=at
     position 0.5 with {\arrow{>}}},postaction={decorate}, xshift=4cm]
\coordinate (A0) at (0,0);
\coordinate (A1) at (1,0);
\coordinate (A2) at (2,0);
\coordinate (B0) at (0,1);
\coordinate (B1) at (1,1);
\coordinate (B2) at (2,1);
\draw[yellow, line width = 3pt] (A0) -- (B0);
\draw[yellow, line width = 3pt] (A1) -- (B1);
\draw[yellow, line width = 3pt] (A2) -- (B2);
\draw[darkgreen] (A0) -- (B1);
\draw[darkgreen] (A1) -- (B2);
\draw[darkgreen] (A2) -- (B0);
\draw[yellow, line width = 3pt] (A0) -- (B2);
\draw[yellow, line width = 3pt] (A1) -- (B0);
\draw[yellow, line width = 3pt] (A2) -- (B1);
\end{scope}
\begin{scope}[yscale = {1}, xscale={1},decoration={markings, mark=at
     position 0.5 with {\arrow{>}}},postaction={decorate}, xshift=7cm]
\coordinate (A0) at (0,0);
\coordinate (A1) at (1,0);
\coordinate (A2) at (2,0);
\coordinate (B0) at (0,1);
\coordinate (B1) at (1,1);
\coordinate (B2) at (2,1);
\draw[darkblue] (A0) -- (B0);
\draw[darkblue] (A1) -- (B1);
\draw[darkblue] (A2) -- (B2);
\draw[yellow, line width = 3pt] (A0) -- (B1);
\draw[yellow, line width = 3pt] (A1) -- (B2);
\draw[yellow, line width = 3pt] (A2) -- (B0);
\draw[yellow, line width = 3pt] (A0) -- (B2);
\draw[yellow, line width = 3pt] (A1) -- (B0);
\draw[yellow, line width = 3pt] (A2) -- (B1);
\end{scope}

\begin{scope}[yscale = {1}, xscale={1},decoration={markings, mark=at
     position 0.5 with {\arrow{>}}},postaction={decorate}, xshift=1cm, opacity =1]
\coordinate (A0) at (0,0);
\coordinate (A1) at (1,0);
\coordinate (A2) at (2,0);
\coordinate (B0) at (0,1);
\coordinate (B1) at (1,1);
\coordinate (B2) at (2,1);
\draw[darkblue] (A0) -- (B0);
\draw[darkblue] (A1) -- (B1);
\draw[darkblue] (A2) -- (B2);
\draw[darkgreen] (A0) -- (B1);
\draw[darkgreen] (A1) -- (B2);
\draw[darkgreen] (A2) -- (B0);
\draw[darkred] (A0) -- (B2);
\draw[darkred] (A1) -- (B0);
\draw[darkred] (A2) -- (B1);
\end{scope}
\begin{scope}[yscale = {1}, xscale={1},decoration={markings, mark=at
     position 0.5 with {\arrow{>}}},postaction={decorate}, xshift=4cm, opacity =1]
\coordinate (A0) at (0,0);
\coordinate (A1) at (1,0);
\coordinate (A2) at (2,0);
\coordinate (B0) at (0,1);
\coordinate (B1) at (1,1);
\coordinate (B2) at (2,1);
\draw[darkblue] (A0) -- (B0);
\draw[darkblue] (A1) -- (B1);
\draw[darkblue] (A2) -- (B2);
\draw[darkgreen] (A0) -- (B1);
\draw[darkgreen] (A1) -- (B2);
\draw[darkgreen] (A2) -- (B0);
\draw[darkred] (A0) -- (B2);
\draw[darkred] (A1) -- (B0);
\draw[darkred] (A2) -- (B1);
\end{scope}
\begin{scope}[yscale = {1}, xscale={1},decoration={markings, mark=at
     position 0.5 with {\arrow{>}}},postaction={decorate}, xshift=7cm, opacity =1]
\coordinate (A0) at (0,0);
\coordinate (A1) at (1,0);
\coordinate (A2) at (2,0);
\coordinate (B0) at (0,1);
\coordinate (B1) at (1,1);
\coordinate (B2) at (2,1);
\draw[darkblue] (A0) -- (B0);
\draw[darkblue] (A1) -- (B1);
\draw[darkblue] (A2) -- (B2);
\draw[darkgreen] (A0) -- (B1);
\draw[darkgreen] (A1) -- (B2);
\draw[darkgreen] (A2) -- (B0);
\draw[darkred] (A0) -- (B2);
\draw[darkred] (A1) -- (B0);
\draw[darkred] (A2) -- (B1);
\end{scope}
\begin{scope}[yscale = {1}, xscale={1},decoration={markings, mark=at
     position 0.5 with {\arrow{>}}},postaction={decorate}, xshift=4cm, yshift = 2cm, opacity =1]
\coordinate (A0) at (0,0);
\coordinate (A1) at (1,0);
\coordinate (A2) at (2,0);
\coordinate (B0) at (0,1);
\coordinate (B1) at (1,1);
\coordinate (B2) at (2,1);
\draw[darkblue] (A0) -- (B0);
\draw[darkblue] (A1) -- (B1);
\draw[darkblue] (A2) -- (B2);
\draw[darkgreen] (A0) -- (B1);
\draw[darkgreen] (A1) -- (B2);
\draw[darkgreen] (A2) -- (B0);
\draw[darkred] (A0) -- (B2);
\draw[darkred] (A1) -- (B0);
\draw[darkred] (A2) -- (B1);
\end{scope}